\documentclass[10pt,a4paper]{amsart}

\usepackage{setspace}
\usepackage[utf8x]{inputenc}
\usepackage[T1]{fontenc}

\usepackage{helvet}

\usepackage[geometry]{ifsym}
\usepackage{amsthm}
\usepackage{amsmath}
\usepackage{verbatim}
\usepackage{amssymb} 
\usepackage{resizegather}
\usepackage{tikz}
\usepackage{tikz-cd}
\usepackage{babel} 
\usepackage{mathrsfs}
\usepackage{graphicx}
\usepackage{epstopdf}
\usepackage{epigraph}

\usepackage{pgfplots}
\usepgfplotslibrary{fillbetween}
\pgfplotsset{compat=1.13}

\usepackage[bottom]{footmisc}

\usepackage{hyperref}

\newtheorem{theorem}{Theorem}

\newtheorem{lemma}{Lemma}[theorem]

\newtheorem{example}{Example}
\newtheorem{definition}{Definition}

\newtheorem*{Acknowledments}{Acknowledments}

\title{On Dynamical Systems in the Étale Topology}
\author{Lars Andersen}

\begin{document}

\begin{abstract}
We discuss $\mathcal{D}$-modules and dynamical systems in the étale topology. We construct differential schemes associated to morphisms $f: X\to S$ of schemes of the same dimension. We introduce differential inertia group $I_{diff}^i$ which act trivially on the special fibers of these schemes. As an application we discuss the problem of counting points on elliptic curves which we show is connected to vanishing cycles of singularities and to closed orbits of dynamical systems.
\end{abstract}

\begin{center}
    

\maketitle

\tableofcontents

\text{Classification: }\textbf{14-XX, 94-XX}
\end{center}

\section{Introduction}
In these notes we make an assault on dynamical systems in positive characteristics. In this framework a solution would be a geometric point of the differential scheme of the dynamical system. There follows a discussion on the number of points of elliptic curves which we formulate as being calculable from the number of orbits of a dynamical system or equivalently as a sum of Milnor numbers of a certain finite family of morphisms. A change in the number of vanishing cycles would correspond to a bifurcation of the dynamical system. The main result is in fact very trivial but in a sense pleasing theoretically in that number theory, analysis and geometry are all present and connected. 

\begin{Acknowledments}
    My gratitude goes to Jan Stevens, Georges Comte and Michel Raibaut who were my master thesis respectively doctoral thesis supervisor.
\end{Acknowledments}

\section{Differential Schemes}
\subsection{The étale topology} Let $k$ be $\mathbb{Z}/p\mathbb{Z}$ for some prime $p$. Let $Sch^{ét}(k)$ denote the category of étale schemes \cite{SGA4} over $k$. These are schemes as ringed spaces in the sense of EGA but with the étale topology instead of the Zariski topology. Let $X\in Sch^{ét}(k)$ be of finite type and irreducible. A point of $X$ is the inclusion $\text{Spec}(\mathfrak{m})\to \text{Spec}(X)$ where $\mathfrak{m}\subset \mathcal{O}_{X}$ is a maximal ideal. 
A geometric point of $X$ is a map $\text{Spec}(k^{alg})\to X$. 

\subsection{Differential schemes} Let $\mathcal{D}=A^{n}(k)$ be Weyl algebra consisting of differential operators 
$$P=\sum_{\alpha} f_{\alpha} \partial_{\alpha}$$
where $f_{\alpha}\in \mathcal{O}_X$ and where $\partial_{\alpha}: \mathcal{O}_X\to \mathcal{O}_X$ are the partial differential operators. A $\mathcal{D}$-module $M$ is a left $A^{n}(k)$-module or equivalently a $\mathcal{O}_X$-module with a left $\mathcal{D}$-action. 
Given an associative ring $R$ and a left $R$-module $M$ say that a proper submodule $P\subset M$ is a prime submodule \cite{Dauns, Bracic} if $m (r R) \subseteq P$ implies $r\in P$ or $m\in Ann(M/P)$. Let 
$\text{Spec}(M)={N\subset M| N \text{prime submodule of M}}$
be the \emph{prime spectrum} of $M$. Then there exists a topology on $\text{Spec}(M)$ such that if $\nu(N)=ann(M/N)$ then
$$D(N)=\text{Spec}(M)\setminus \nu (N)$$
is a basis of open subsets. Furthermore $\nu: \text{Spec}(M)\to \text{Spec}(R)$ is continuous with respect to this topology which we will following\cite{Bracic} call the Zariski topology on the prime spectrum of modules. We can remark here that $\text{Spec}(M)$ might be empty. Suppose now that \textit{$R$ is such that there is a choice of $N$ as above satisfying the left Ore conditions with respect to the set of $\nu(N)$}.\\

\begin{definition}
    Let $\Gamma(D(N), \mathcal{O}_M)=M\otimes_R [\nu(N)^{-1}]R$ where $[\nu(N)^{-1}]R$ is the left localisation of $R$ by $\nu(N)$.
\end{definition}

Then $(M, \mathcal{O}_M)$ is a ringed space.

\begin{definition} A differential affine scheme is any scheme isomorphic to $(M, \mathcal{O}_M)$ as a ringed space where $M$ is a $\mathcal{D}$-module.
\end{definition}

In the case where $M=\text{Spec} k[x_1,\dots, x_n]$ one knows \cite[Proposition I.2]{coutinho} that $k[x_1,\dots, x_n]$ is a simple module hence $(0)$ is the only prime submodule. Then
$$(\text{Spec}(k[x_1,\dots, x_n]), [k[x_1,\dots, x_n]\otimes_{\mathbb{D}} \text{Ann}(k[x_1,\dots, x_n])^{-1}]\mathcal{D})=$$
$$((0), [k[x_1,\dots, x_n]\otimes_{\mathbb{D}} \text{Ann}(k[x_1,\dots, x_n])^{-1}]\mathcal{D})$$
which already has a richer structure albeit that it consists of a point topologically. Replacing the module in questions with algebraic formal power series we would then get richer and richer generic points.\\
Note however that $\text{Spec}(\mathbb{D})$ is not a point; the Weyl algebrs $\mathbb{D}$ is simple as a ring but not as a module over itself. 

\begin{definition}
    A morphism of differential schemes is a morphism of ringed spaces.
\end{definition}

\subsection{The differential inertia groups}
We propose to define an analogue of the inertia groups from number theory. In number theory these groups gives information on the ramification of primes such as wild ramification. In geometry they provide according to Deligne \cite{SGA7II} the number of vanishing cycles in the arithmetic case where one has to not only take into account the usual vanishing cycles but also wild vanishing cycles. The latter only appear when the residue characteristics is positive whereas the tame vanishing cycles corresponds to the usual vanishing cycles. The case to consider is when one has e.g. 
$$f(x, y)=y^3+x^2+x^3,\qquad (x, y)\in (\mathbb{Z}/2\mathbb{Z})^2$$
Then the origin is an isolated point but the Jacobian ideal is of dimension $4$. Over the complex numbers the dimension is instead $2$. Thus, there are $2$ wild and $2$ tame vanishing cycles. This is of course due to the behaviour of tangent planes since the differential of a map depends on the residue characteristics. For differential equations one has the same phenomenon yet higher differentials plays a role. We therefore propose the following definition.
\begin{definition}
  The differential inertia groups are the groups
  $$I_{diff}^i=\{D\in \mathcal{D}_X/ \mathcal{O}_X| D \partial^k\equiv 0  \quad\forall k\leq i\}$$ 
  where we write $\equiv$ to mean that only constant solutions to the corresponding differential operator can be found.
\end{definition}

Note that this means that there is only the trivial point (in the sense of the functor of points) of the $\mathbb{D}$-module namely the trivial solution.

\begin{example}
On $X=\text{Spec}\mathbb{Z}/p\mathbb{Z}[x]/x^4$ we have $\partial^2\partial (1+x+x^2+x^3)=6$
so $D=\partial$ is in $I_{diff}^2$ in any characteristics. If $p=3$ then  $D\in G_D^1$. If $p=2$ then $D D(1+x+x^2+x^3)=D^1(1+3x^2)=6x=0$ so $D\in G_D^1$.
\end{example}

Let $(M, \mathcal{O}_M)$ be a differential scheme defined over $(X, \mathcal{O}_X)$. Then the differential inertia groups $I_{diff}^i$ act on the constant sheaf $\mathcal{D}_X$ hence on the $\mathcal{D}_X$ module $M$ and thus on $\Gamma(U, \mathcal{O}_M)$ by construction. On the other hand at a geometric point $\bar{x}=\text{Spec}(k(x)^{alg})$ whose local ring has generic point $\eta$ and geometric generic point $\bar{\eta}=\text{Spec}(k(\eta)^{alg})$ 
the inertia group $I$ is defined by
$$1\to I \to \text{Gal}(k(\eta)^{alg}/k(\eta))\to \text{Gal}(k(s)^{alg}/k(s))\to 1$$
The inertia group therefore acts trivially on $k(x)^{alg}$ and thus also trivially on sheaves on $X$ localised at $\bar{x}$. Given a morphism of schemes $f: X\to S$ write $X_s=X\times_S s$ for the scheme-theoretical fiber over $s$.
\begin{definition}
The sheaves of vanishing cycles of a morphism of schemes $f:X\to S$ at a point $\bar{x}$ is the right derived functor 
$$R^i\Psi(K)_{\eta})=R^i \bar{i}^{\ast} \bar{j}_{\ast}(K_\eta)$$
where $\bar{i}: X_{\bar{s}}\to \bar{X}$ and $\bar{j}: X_{\bar{\eta}}\to \bar{X}$ are the inclusions of the geometric special fiber respectively geometric generic fiber into $\bar{X}=X\times_S \bar{S}$ where $\bar{S}\to S$ is a normalisation of $S$ at $k(\eta)$.
\end{definition}

The above discussion means the action of the inertia group is trivial on the sheaves of vanishing cycles. 

Our goal is to connect these two different actions in the case where $f: X\to S$ is a morphism of smooth schemes such that $f$ has an isolated critical point in the origin $s$. We therefore propose to study 
\begin{definition}
    The associated $\mathcal{D}_S$-module is $V=f-\partial/\partial s$
\end{definition}
\begin{definition}
    The gradient associated $\mathcal{D}_{S}$-module is $W=\text{grad} f-\partial/\partial s$. 
\end{definition}

The associated can be localised with localisations
$$V_{s}=f_s- \partial/\partial s,\qquad V_{\eta}=f_{\eta}-\partial/\partial s$$
and similarly for $W.$ 
\subsubsection{} The inertia group acts on the coefficients and acts trivially on $V_{\bar{s}}$. The differential inertia group $I_{diff}^i$ on $S$ is
$I_{diff}^i=\{D| \partial_s^k D \equiv 0,\quad\forall k\leq i\}$
where $\partial_s=\partial/\partial s$.

Let $k=0$. If $D=\partial_s \in I_{diff}^{0}$ then 
$\partial x/ \partial s=0$ so there is only the constant solution.
Let $k=1$. If $D=\partial_s \in I_{diff}^{1}$ then 
$\partial^2 x/ \partial s^2=0$ and $(\partial^2 f/\partial x_i^2)(\partial x_i/\partial s)^2=0$ which gives $\partial^2 f/\partial x_i c_i^2=0$.\\
Let $k=2$. If $D=\partial_s \in I_{diff}^{2}$ then 
$\partial^3 x/ \partial s^3=0$ and 
$$\frac{\partial f}{\partial x_i}(c_i t^2+c_i' t+ c_i'')+3c_i(c_it+c_i') f=k_i(x')+c(x')$$
with solutions
$$f=-k_i(x')\frac{c_i t^2+ c_i' t+c_i''}{9c_i^2(c_it+c_i')^2}$$
$$+k_i(x')\frac{x_i}{9c_i(c_it+ c_i')^2}+\frac{c(x')}{3c_i(c_i t+ c_i')}$$
which is clearly not a polynomial function anymore. We have proven

\begin{lemma}\label{alemma}
    The differential inertia groups act trivially on $V_s$.
\end{lemma}
\subsubsection{} The inertia group $I$ acts on the coefficients and acts trivially on $W_{\bar{s}}$. We have
$\frac{\partial}{\partial s} \frac{\partial f}{\partial x_i} (x_i (s))= \frac{\partial^2 f}{\partial x_i^2} \frac{\partial x_i}{\partial s}.$
Let $k=0$. If $D=\partial_s \in I_{diff}^{0}$ then 
$\partial x/ \partial s=0$ and the constant solution is the only solution. Let $k=1$. If $D\in I_{diff}^1$ then $\partial^2 x/\partial s^2=0$ and 
$$\frac{\partial^2 f}{\partial x_i^2}\frac{\partial x_i}{\partial s}=k_i$$
and 
$$f= \frac{k_i}{m_i} x_i^2+d_i(x') x_i +e_i(x')$$
for each $i=1,\dots, n$. By symmetry we find that 
$$f=\sum c_i x_i^2$$
and in particular the differential inertia groups stabilise for $k\geq 2$. If the characteristics divide the coefficients of $f$ we get more erratic  behaviour. Let $k=2$. If we consider the action of $I_{diff}^k$ on $W_{\eta}$ instead then we reach the same conclusions as before for $k=0,1$. If $k=2$ and if $D\in I_{diff}^2$ then $\partial^3 x/\partial s^3=0$ and 
$$\frac{\partial^3 f}{\partial x_i^3}=\frac{k_i}{c_i(c_i t+ l_i)^2}$$
as a map $\pi_i(k(\eta))\to k$. In particular the action is non-trivial in general. We have proven 
\begin{lemma}
    In the situation above, the differential inertia groups acts trivially on $W_s$ for all $k\geq 1$ unless $f$ is a Morse singularity.
\end{lemma}

\subsection{The sheaf of differential cycles}
Consider $f: X\to S$ a morphism of smooth schemes of equal dimensions. A point on a differential scheme $(\text{Spec}(M), \mathcal{O}_M)$ where $M$ is a $\mathcal{D}_X$-module is a morphism of ringed spaces $\text{Spec}\mathcal{O}_{X, x}^{sh}\to \text{Spec}(M)$ where $x\in X$ and where the ring of algebraic power series $\mathcal{O}_{X, x}^{sh}$ is endowed with the structure of a $\mathcal{D}_X$-module by the standard differentials in an étale neighborhood of $x$. We will write $X_{(x)}=\text{Spec} \mathcal{O}_{X,x}^{sh}$ and $\tilde{\eta}=\text{Spec} Frac \mathcal{O}_{X,x}^{sh}$.
Recall \cite{SGA7I} that the étale Milnor fiber at $\bar{x}$ is $X_{(x)} \times_{\tilde{\eta}} \bar{\eta}$. 
Let $T X$ be the tangent space of $X$ over $S$ and consider the commutative diagram
\[ \begin{tikzcd}
\mathcal{V}  \arrow{r}{proj}\arrow{dr}\arrow[swap]{d}{(x_1^{-1}, \dots, x_n^{-1})} & X \arrow{d}{f} \\%
S \arrow{r}{\partial/\partial s}& TS
\end{tikzcd}
\]
with the diagonal morphism being $f\times_S \partial/\partial s$ and the upper morphism coming from the inclusion $X\to T X$. 
\begin{definition} 
 The differential scheme of $f$ is 
 $\mathcal{V}\in Sch^{ét}$ with fibers $\mathcal{V}_{s}$ and $\mathcal{V}_{\eta}$ given by the forming commutative diagrams as above but with $X$ replaced by $i: X_s\to X$ and $j: X_{\eta}\to X$.
 \end{definition}
 Let $x\in V(f)$ and let $s$ denote the origin in $S$. Consider the commutative diagrams
 \[ \begin{tikzcd}
\mathcal{V}_{(s)} \arrow{r}{proj} \arrow{dr} \arrow[swap]{d}{(x_1^{-1},\dots, x_n^{-1})} & X_{s} \arrow{d}{f} \\%
S_{\bar{s}} \arrow{r}{\partial/\partial s}& T S_{\bar{s}}
\end{tikzcd}
\]
and 

\[ \begin{tikzcd}
\mathcal{V}_{\bar{\eta}} \arrow{r}{proj} \arrow{dr} \arrow[swap]{d}{(x_1^{-1},\dots, x_n^{-1})} & X_{\bar{\eta}} \arrow{d}{f} \\%
S_{\bar{\eta}} \arrow{r}{\partial/\partial s}& T S_{\bar{\eta}}
\end{tikzcd}
\]

In the fiber product $\mathcal{V}=X\times _{TS} S$ we have the lifts 
$$\bar{i}: X\times_{TS} S\times\bar{s} \to  X\times_{TS} S$$ 
and 
$$\bar{j}: X\times_{TS} S\times \bar{\eta}\to X\times_{TS} S$$ 
of the inclusions $i: \text{Spec}(k(\bar{s}))\to TS$ and $j: \text{Spec}(k(\bar{\eta}))\to TS$. Let $r: S\to k$ be $r(x_1,\dots, x_n)=\sum x_i^2$ and consider instead $i^r=r^{-1}\circ \tilde{i}: \text{Spec}(k(s^r)\to S$ and $j^r=r^{-1}\circ \tilde{j}: \text{Spec}(k(\eta^r)\to S$ with $\tilde{i}, \tilde{j}$ denoting the inclusions of the closed and generic point of the ground field. 

 \begin{definition}
 The differential sheaf of $K\in D(X)$ at $x$ is $$\zeta= \bar{i}^{r \ast} \bar{j}^r_{\ast}(K_{\eta})$$
\end{definition}

We will now define closed orbits as being $\mathbb{S}^1$-valued points of
$V_{x, s}=f_x\times_S \partial/\partial s$ over $s$ where $f_x$ denotes the strict localisation of $f$. In detail
\begin{definition}
    The complex of differential cycles or closed orbits of $f$ at $x\in V(f)$ is the derived complex of sheaves $R^{\ast} \zeta(K_{\eta})$.
\end{definition}

The main theorem is then the following

\begin{theorem}\label{atheo} Suppose that the characteristics is distinct from $2$. The stalks of $R^{\ast} \zeta(K_{\eta})$ over $(x, s), x\in V(f)$ are generated by vanishing cycles of $r\circ f$ at $x$. In particular their number is given by the generalized Milnor number, by Deligne's Theorem \cite{SGA7II}. The stalks are also generated by the periodic solutions of the dynamical system $f_s=\partial/\partial s$.
\end{theorem}
\begin{proof}
The stalks $R^{\ast} \zeta(K_{\eta})_{x, \eta^r}$ are 
$$H^{\ast}(r\circ f_x)\times_k (r\circ \partial/\partial s) \times_k \bar{\eta}^r, K)=H^{\ast}((f_x-\partial/\partial s)^{-1}(r^{-1}\eta^r))$$
which if $char\neq 2$ are the  periodic solutions to $f_s=\partial/\partial s$ or the vanishing cycles of $r\circ f$ in the sense of \cite{SGA7II}.
\end{proof}

\begin{lemma} The differential inertia groups act trivially on $R^{\ast} \zeta(K_{\eta}$
\end{lemma}
\begin{proof}
    See the discussions above and \hyperref[alemma]{Corollary \ref*{alemma} }.
\end{proof}

\begin{example} Consider the dynamical system 
$$dx/dt= y^2+ax^3+bx,\quad dy/dt = y - l$$
over $k$ the $p$-adic numbers for $p$ a large prime. We let 
$$f_i: \text{Spec} \mathbb{Z}/p\mathbb{Z}[x, y]\to \mathbb{Z}/p\mathbb{Z},$$
$$f_i(x, y)=i^2+ax^3+bx$$
with critical locus $\Sigma_i$
and let $l_i=\bigoplus_{j\in \Sigma_i} \dim H^0(\mathcal{F}_i^{ét})$ denote the sum of dimensions of the homology of the étale Milnor fibers of $f_i$. Then 
$$|E(\mathbb{Z}/p\mathbb{Z})|=\sum_{i=1}^p l_i +1$$
where $E(\mathbb{Z}/p\mathbb{Z})$ denotes the \emph{set} of points of $\mathcal{C}=\{y^2+ax^3+bx=0\}$. In fact since this set is finite,
$$|E(\mathbb{Z}/p\mathbb{Z})|=\sum_{i=0}^p \mathcal{C}\cap \{y=i\}(\mathbb{Z}/p\mathbb{Z})+1=\sum_{i=0}^p l_i+1$$
because at each point there is either one, two or three vanishing cycles according to whether the cubic equation defining the summand has one, two or three points so we get the result in light of the \hyperref[atheo]{Theorem \ref*{atheo} }. Here we added the point at infinity. 
\end{example}

\section{Appendix: Collatz Conjecture}
An outline of a proof would be to consider Collatz function as a continuous $2$-adic function $f: \mathbb{A}_k^1\to \mathbb{A}_k^1$.
The associated dynamical system has a stationary point at the origin $x_1=0$ and a stationary point at $x_2=1$. Corresponding to $x_2$ is unique tame vanishing cycle, since the local equation $3x+1$ has only tame ramification. Corresponding to $x_1$ is a map singularity; the map $x\mapsto x/2$ isn't defined over the residue field $k(x_1)=\mathbb{Z}/2\mathbb{Z}.$ Now, if there is a closed orbit then there is a vanishing cycle at $x_1$. This vanishing cycle would be in an étale neighborhood of the origin. Following \cite{SGA7II} it would be an element of 
$$f^{-1}(\text{Spec}\bigcup_l \text{Frac }\mathbb{F}_2^l((x)))\cap \text{Spec} \bigcup_l\mathbb{F}_2^l[[x]]$$
hence a solution in terms of a zero of an algebraic equation with $\mathbb{Z}/2\mathbb{Z}$-coefficients. So one possible proof of the conjecture would be to argue in detail that there can or cannot be any such solution due to $f$ not being continuous reduced modulo $2$.

\bibliographystyle{plain}
\bibliography{main.bib}

\end{document}